\documentclass[12pt,a4paper,english,reqno]{amsart}
\usepackage[a4paper, margin= 1in]{geometry}
\usepackage{amsmath}
\usepackage{amssymb}
\usepackage{amsthm}
\usepackage{graphicx}
\usepackage{gensymb}
\usepackage{matlab-prettifier}
\usepackage[english]{babel}
\usepackage{tikz}
\usepackage{cite}
\usepackage{hyperref}
\usepackage{float}

\newtheorem{theorem}{Theorem}[section]
\newtheorem{corollary}[theorem]{Corollary}
\newtheorem{question}[theorem]{Question}
\newtheorem{conjecture}[theorem]{Conjecture}
\title{A Covering Pursuit Game}
\author{Benjamin Gillott}
\begin{document}
\maketitle
\setcounter{section}{-1}

\begin{abstract}
In the `Covering' pursuit game on a graph, a robber and a set of cops play alternately, with the cops each moving to an adjacent vertex (or not moving) and the robber moving to a vertex at distance at most 2 from his current vertex. The aim of the cops is to ensure that, after every one of their turns, there is a cop at the same vertex as the robber. How few cops are needed? 
\\
\\
Our main aim in this paper is to consider this problem for the two-dimensional grid $[n]^2$. Bollob\'{a}s and Leader asked if the number of cops needed is $o(n^2)$. We answer this question by showing that $n^{1.999}$ cops suffice. We also consider some applications. In particular we study the game `Catching a Fast Robber', concerning the number of cops needed to catch a fast robber of speed $s$ on the two-dimensional grid $[n]^2$. We improve the bounds proved by Balister, Bollob\'{a}s, Narayanan and Shaw for this game.
\end{abstract}

\section{Introduction}
The game of Cops and Robbers, introduced over forty years ago independently by Nowakowski and Winkler \cite{og1} and Quilliot \cite{og2}, is a perfect information pursuit-evasion game played on an undirected graph $G$. The cops initially position themselves on the vertices of the graph as they wish, the robber then choses a vertex to start on. They then take turns moving; the cops going first. On the cops' turn each cop may move to an adjacent vertex or stay put, and similarly for the robber. The cops win if at some point they are able to occupy the same vertex as the robber. The robber wins if he is able to avoid the cops indefinitely. The cop number of a graph, $c(G)$, is defined to be the minimum number of cops needed to win.
\\
\\
In this paper we study the related game of Covering. Covering is played on a graph $G$,  the robber can move at a faster speed than the cops, $2$ unless otherwise stated. In other words, the robber may move to a vertex at most distance $2$ away on each turn of his, whilst each cop can only move to an adjacent vertex or stay put. The cops choose the starting configuration. The team of cops win if they are able to ensure that after each turn of theirs at least one cop occupies the same vertex as the robber. For a given graph $G$, we ask how many cops are needed to win Covering on $G$.
\\
\\
To get a feel for this game, let us first consider Covering played on the path, $[n] = \{1,2,3,..,n\}$. Here the exact answer is $\lceil\frac{n}{2}\rceil$. To see that this many cops suffices, the $i$th cop can guard $\{2i-1,2i\}$ (or just $\{2i-1\}$ if $2i > n$) by starting at $2i-1$ and each turn staying put if the robber is not on $\{2i-1,2i\}$ or jumping to the robber if he is. The robber can beat $\lceil\frac{n}{2}\rceil - 1$ cops by initially moving to $1$, and then `sprinting' to the right moving $2$ to the right each turn. We see that no cop could be on top of the robber at two different times whilst he is `sprinting', therefore there must have been an instant where no cop was on the robber.
\\
\\
What happens in two dimensions? Thus our graph is $[n]^2$ with distance given by the $l_\infty$ metric, i.e. two distinct vertices are adjacent if each co-ordinate differs by at most 1. The two-dimensional analogues of these strategies give linear lower bounds and quadratic upper bounds. Indeed, each cop is able to guard a $2$ by $2$ sub-square so $\lceil\frac{n}{2}\rceil^2$ cops suffice. The robber can sprint along the main diagonal beating $\lceil\frac{n}{2}\rceil - 1$ cops. 
\\
\\
This motivated Bollob\'{a}s and Leader \cite{communication} to ask if $o(n^2)$ cops could win Covering played on $[n]^2$. Section $1$ is primarily concerned with this question. Our main result of the paper shows that this does hold, in particular in Theorem 1.2 we show that $n^{1.999}$ cops can win. We also give a lower bound of $n^{1.357}$ in Theorem 1.1. Both of these results are for sufficiently large $n$.
\\
\\
In Section $2$ we consider the natural generalisation to Covering on $[n]^d$. It turns out that a related notion of key importance is exact capture time (meaning that the cops do not have to cover the robber at all times, but only at a fixed pre specified time $T$). We prove in Theorem 2.1 that there are constants $c_1$ and $c_2$ such that it takes between $c_1T^{d/2}$ and $c_2T^{d/2}$ cops to ensure that a robber is caught at time exactly $T$  in the $d$ dimensional grid. We use this to show, in Corollary 2.2, that the number of cops needed to win Covering on $[n]^d$ is between $n^{\frac{d}{2}}$ and $n^{\frac{d}{2}+1}$. 
\\
\\
We briefly touch on Covering on the torus at the end of Section 1; the torus being the graph on $[n]^2$, with an edge between two distinct vertices if each coordinate differs by at most 1 modulo $n$. Perhaps surprisingly, in Theorem 1.3 we show that the answer on the torus differs from the answer on the grid by at most a factor of $4$.
\\
\\
Lastly, Sections $3$ and $4$ are concerned with applications of Covering, namely to the games of  Catching a Fast Robber and Rugby. Catching a Fast Robber was introduced by Fomin, Golovach, Kratochvıl, Nisse, and Suchan in \cite{ogfastrobber}. The game is concerned with how many cops are needed to catch a fast robber on the grid $[n]^2$. The robber moves with a speed $s >1$, the cops choose their starting position and then the robber chooses where to start. The robber loses if at some point a cop is on the same vertex as the robber. In Corollary 4.1 we show an upper bound of $n^{0.999}$ for a robber with speed 2. This is the first $o(n)$ upper bound. In Corollary 4.2 we prove a lower bound of $n^{0.03}$ for a robber with speed 2. This improves on the lower bound proved in \cite{BALISTER2017341} by Balister, Bollob\'{a}s, Narayanan and Shaw of $e^{c\log(n)/\log(\log(n))}$ for a robber with speed $10^{25}$.
\\
\\
Rugby is played on $\mathbb{Z} \times [n]^2$ - distance being given by $l_\infty$ again - and the robber being allowed to move at most a distance $2$ each turn. The cops choose their initial position, then the robber starts at $(-C, n, n)$, where $C$ is very large so that the robber's first coordinate is at least $2n$ less than any cops' first coordinate. The robber wins if he is able to pass the plane $\{k\} \times [n]^2$ for all positive $k$, without any cop ever being on top of him. We ask how many cops are needed to prevent the robber from winning. It is worth noting that if all cops stay in the same plane in Rugby, the question reduces to Covering on $[n]^2$. We primarily look at the $1$-dimensional game of Rugby, played on $\mathbb{Z}\times [n]$, as this still captures the difficulty of the problem. We find that the $1$-dimensional game of Rugby is remarkably similar to Catching a Fast Robber. In Theorem 3.2 we prove an upper bound of $n^{0.999}$ and in Theorem 3.3 we prove a lower bound of $n^{0.03}$ for sufficiently large $n$. Adapting these, rather straightforwardly, to Catching a Fast Robber is how we prove Corollaries 4.1 and 4.2.
\\
\\
In all sections, results proved for a robber with speed 2 adapt to a robber with speed $s$, $s>1$, giving different bounds. For background on pursuit evasion games see William Baird and Anthony Bonato \cite{background}.
\section{Covering on $[n]^2$}
We start this section by giving a lower bound on the number of cops needed. We define the `density' of a $k$ by $k$ square to be the number of cops in it divided by $k^{1.357}$. The strategy is to seek out smaller sub-squares with no larger density and move into them, and then seek a smaller square again. This way, at some point we enter a sub-square small enough such that there are no cops in it (so long as the initial density was small enough).
\\
\\
What we need to ensure is that the cops outside the sub-square when we entered it cannot reach us - we do this by restricting where we look for the sub-squares. We also need to ensure that the squares we look at moving into are all `disjoint', in the sense that a cop could not have been in multiple of them. One can see that the squares in Figure 1 have these properties, when the robber sprints (i.e. runs at full speed) up the diagonal. In the theorem below `we' will always mean the robber.

\begin{theorem}
We need $\Omega(n^{1.357})$ cops to win Covering on the Grid, $[n]^2$.
\end{theorem}

\begin{proof}
In the proof we will make it easier for the robber by changing the rules slightly so that the cops need to be on them after the robber's moves as well - this only changes the answer by a factor of at most $25$, since if one has a cop strategy for the original game then by adding $24$ cops around each cop on the centred 5 by 5, this wins the modified game as well.
\\
\\
We prove the analogous result for $n^{1.2 -\epsilon}$, any $\epsilon > 0$, initially for clarity. What we show is that for a robber starting in the corner of a  square of side length $N = 2^{n} - 33$, with cop density (defined with $1.2$ instead of $1.357$) at most $2^{-22}$, it is possible for the robber to reach an empty square, not reachable by any cops initially outside the $N$ by $N$ square. The base cases, $n \leq 18$, are immediate as there are no cops in the $N$ by $N$ square. Once we have shown this we have the result for any size square as the robber can play on the largest $2^{n} -33$ sided square contained in his own.
\\
\\
We note that the choice of $33$ is simply so we never have to worry about whether `the cops are too close for our argument to work'.
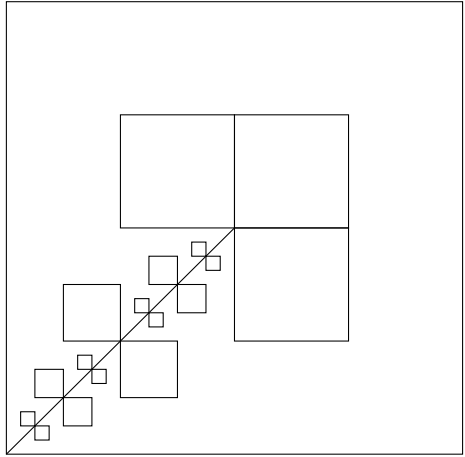
\begin{figure}[H]
\centering
\begin{tikzpicture}
\draw (0,0) -- (3,3);
\draw (0,0) -- (0,6) -- (6,6) -- (6,0) -- cycle;
\draw (3,3) -- (3,4.5) -- (4.5,4.5) -- (4.5,3) -- cycle;
\draw (3,3) -- (3,4.5) -- (1.5,4.5) -- (1.5,3) -- cycle;
\draw (3,3) -- (4.5,3) -- (4.5,1.5) -- (3,1.5) -- cycle;
\draw (1.5,1.5) -- (1.5,2.25) -- (0.75,2.25) -- (0.75,1.5) -- cycle;
\draw (1.5,1.5) -- (2.25,1.5) -- (2.25,0.75) -- (1.5,0.75) -- cycle;
\draw (3/4,3/4) -- (3/4,9/8) -- (3/8,9/8) -- (3/8,3/4) -- cycle;
\draw (3/4,3/4) -- (9/8,3/4) -- (9/8,3/8) -- (3/4,3/8) -- cycle;
\draw (1.5+3/4,1.5+3/4) -- (1.5+3/4,1.5+9/8) -- (1.5+3/8,1.5+9/8) -- (1.5+3/8,1.5+3/4) -- cycle;
\draw (1.5+3/4,1.5+3/4) -- (1.5+9/8,1.5+3/4) -- (1.5+9/8,1.5+3/8) -- (1.5+3/4,1.5+3/8) -- cycle;
\draw (3/8,3/8) -- (3/8,9/16) -- (3/16,9/16) -- (3/16,3/8) -- cycle;
\draw (3/8,3/8) -- (9/16,3/8) -- (9/16,3/16) -- (3/8,3/16) -- cycle;
\draw (3/8+3/4,3/8+3/4) -- (3/8+3/4,9/16+3/4) -- (3/16+3/4,9/16+3/4) -- (3/16+3/4,3/8+3/4) -- cycle;
\draw (3/8+3/4,3/8+3/4) -- (9/16+3/4,3/8+3/4) -- (9/16+3/4,3/16+3/4) -- (3/8+3/4,3/16+3/4) -- cycle;
\draw (3/8+1.5,3/8+1.5) -- (3/8+1.5,9/16+1.5) -- (3/16+1.5,9/16+1.5) -- (3/16+1.5,3/8+1.5) -- cycle;
\draw (3/8+1.5,3/8+1.5) -- (9/16+1.5,3/8+1.5) -- (9/16+1.5,3/16+1.5) -- (3/8+1.5,3/16+1.5) -- cycle;
\draw (3/8+9/4,3/8+9/4) -- (3/8+9/4,9/16+9/4) -- (3/16+9/4,9/16+9/4) -- (3/16+9/4,3/8+9/4) -- cycle;
\draw (3/8+9/4,3/8+9/4) -- (9/16+9/4,3/8+9/4) -- (9/16+9/4,3/16+9/4) -- (3/8+9/4,3/16+9/4) -- cycle;
\end{tikzpicture}
\caption{The possible sub-squares to enter}
\end{figure}
Let $\alpha$ be the cop density of the $N$ by $N$ square. Suppose it is true for all smaller $n$. We start in co-ordinate $(1,1)$ and walk up the main diagonal to $(2^{n-4}-3,2^{n-4}-3)$ in our $N$ by $N$ square, always taking steps of length 2. We then enter into the sub-square with side length $2^{n-5}-33$ with bottom right co-ordinate $(2^{n-4}-3, 2^{n-4}-1)$ if the cop density of that sub-square is smaller than the original square's cop density. Similarly, we do the same for the sub-square with top left co-ordinate $(2^{n-4}-1,2^{n-4}-3)$ and side length $2^{n-5}-33$. We can then induct: since no cops from outside the $N$ by $N$ square could have reached either smaller square in this time, we know that if we reach an empty square in either sub-square that is not reachable from outside, then we are done.
\\
\\
If we entered neither of these sub-squares, we go up to co-ordinate $(2^{n-3}-5, 2^{n-3}-5)$ and perform the same strategy with the two sub-squares off the main diagonal as in Figure 1. We continue this until we enter a smaller sub-square and can induct. If this never happens we reach $(2^{n-1}-17,2^{n-1}-17)$ - one square off the centre of the $N$ by $N$ square. 
\\
\\
At this point, if any of the three sub-squares based at the centre, shown in Figure $1$, have density at most $\alpha$ then we can enter one of them and induct. If the three sub-squares are all of higher density than the original square, then we have a contradiction as follows: one can see that if a cop was in one of the sub-squares when we checked the density of it, then it is impossible for this cop to have been in any of the other sub-squares when we checked their density. So we must have at least this many cops:
$$
\alpha(3(2^{n-2}-33)^{1.2} + 2(2^{n-3}-33)^{1.2} + 4(2^{n-4}-33)^{1.2} + 8(2^{n-5}-33)^{1.2})
$$
For $n \geq 19$, this is more than $\alpha N^{1.2}$. This gives us a contradiction.
\\
\\
Now, if we consider repeating the same argument with a finer configuration of sub-squares that we look to enter. We add in $16$ squares of side length $2^{n-6}-33$ between each pair of squares in Figure $1$, then $32$ squares of side length $2^{n-7}-33$, and so on. Note we will need to increase the value of $33$ as well, in order to have that the squares we wish to enter have integer coordinates.
\\
\\
For us to be able to prove a lower bound of $n^f$ for sufficiently large $n$ in this manner, we need that $3(1/4)^f + 2(1/8)^f + 4(1/16)^f + ... > 1$. This is equivalent to $f < \log_2 \frac{1+\sqrt{17}}{2}$, one can check that $1.357$ satisfies this.
\end{proof}

We do not see a way to improve this strategy. What we have been able to show is that this is the furthest one can take it by looking at moving into sub-squares and running up the main diagonal. But it is still plausible that one could do better by running into different sub-shapes or along a different path.
\\
\\
We now move on to showing an upper bound to the number of cops needed. There are two ideas that go into this. Firstly, we observe that we can build a strategy for a bigger square by using strategies for smaller squares. If we have a strategy to catch a robber on an $n$ by $n$ square, then consider a $15n$ by $15n$ square tiled with $225$ $n$ by $n$ sub-squares. On each sub-square we have a team of cops playing our strategy to catch a phantom robber defined to be on the vertex in their sub-square closest to the real robber. This is possible as we observe that this phantom robber never moves a distance more than $2$. Then, when the robber really enters any sub-square, they occupy the position of the phantom robber and therefore a cop is on top of them.

\begin{figure}[H]
\centering
\begin{tikzpicture}
\draw[step=1cm,gray,very thin] (0,0) grid (10,1);
\fill[black] (0,0) rectangle (1,1);
\foreach \x in {1,2,...,9} {
        \draw[->,bend left] (\x+0.4,-0.1) to (\x-0.4,-0.1);
    }
\draw[step=1cm,gray,very thin] (0,-2) grid (10,-1);
\fill[black] (9,-2) rectangle (10,-1);
\end{tikzpicture}
\caption{A fast hole}
\end{figure}
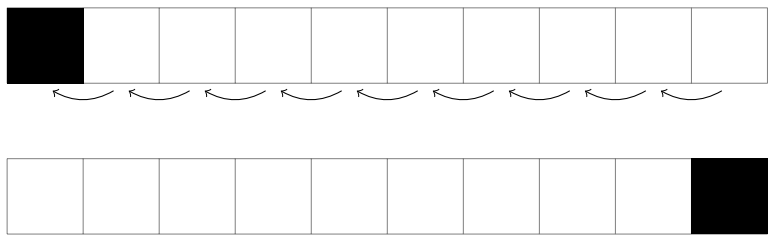

Secondly, we consider trying to do the strategy above but missing one sub-square, we move the teams of cops around to shift the missing sub-square such that the robber can never reach the missing square (or any team of cops thats moving). To achieve this we use that a `hole' moves faster than a cop. If along a path of cops, with a hole at one end, all cops move on step along the path towards the hole, in effect no cop has moved and the hole has shifted to the other end of the path. See Figure 2 for a pictorial representation of this. We move the missing sub-square around in this manner.

\begin{theorem}
$O(n^{1.999})$ cops can win Covering on the Grid, $[n]^2$.
\end{theorem}

\begin{proof}
We show that $224^n$ cops can win on a $15^n$ square by induction on $n$. This then implies the result as $1.999 > \ln(224)/\ln(15)$. The base case is trivial. Let $N = 15^n$ and suppose it is true for $n$. We will have $224$ teams of $224^n$ cops.
\\
\\
We then tile the grid of side length $15N$ with $225$ $N$ by $N$ squares. On each of these sub-squares we have a phantom robber, defined to be on the vertex in our sub-square closest (with euclidean distance) to the real robber. We can see that this phantom robber never moves more than $2$ steps in a real robbers' turn. Therefore, this phantom robber is playing like a real robber and they coincide when the real robber is in the sub-square. We then have a team of $224^n$ phantom cops confined to this sub-square, playing to cover this phantom robber. What we will ensure is that whenever the real robber is in a given sub-square, one of our teams of cops has occupied exactly the phantom cops' positions. This therefore ensures, as the phantom cops always cover the phantom robber, that whatever sub-square the robber is in a real cop is on top of him.
\\
\\
In a `stationary state' each team of cops is occupying a sub-grid's phantom cops' positions, leaving one empty sub-grid either in the top right or the bottom left. We assume that in a stationary state the robber is not in the section of the grid with the empty square, where we have three sections according to the dotted lines $x+y=10N$ and $x+y=20N$. We can place the cops initially so that this is the case.
\\
\\
We then are in a `moving state' if the robber ever enters the section with the empty square: at this point we either move each team one along the bottom and right hand side if the robber is above the line $y=x$; or we move each team of cops one along the top and left hand side if the robber is below the line $y=x$. Individual cops in a team of cops each choose a phantom cop in the square they are moving to copy. They then can catch this phantom cop in at most $2N$ time as they can just simply chase each co-ordinate. After each cop does this we are back in a stationary state with the hole in the other corner.
\\
\\
In Figure 3, the red indicates the robber's position, the black indicates the empty cell and the grey indicates the cells `in motion'.
\\
\\
The numbers are chosen so that it is impossible for the robber to reach a square that has phantom cops without real cops on them in either stage. It is also not possible, as the gap between the diagonals takes the robber $2.5N$ time to traverse, for the robber to initialise a new moving stage while a previous one is not finished. This completes the induction.

\begin{figure}[H]
\centering
\scalebox{1.35}{
\begin{tikzpicture}
\draw[step=0.2cm,gray,very thin] (0,0) grid (3,3);
\fill[black] (0,0) rectangle (0.2,0.2);
\fill[red] (0.6,1.4) rectangle (0.8,1.6);
\draw[black,dashed] (0,2) -- (2,0);
\draw[black,dashed] (3,1) -- (1,3);

\draw[step=0.2cm,gray,very thin] (3.4,0) grid (6.4,3);
\fill[gray!50!white] (3.4,0) rectangle (6.4,0.2);
\fill[gray!50!white] (6.2,0) rectangle (6.4,3);
\fill[red] (0.4+3.4,1.2) rectangle (0.6+3.4,1.4);
\draw[gray,very thin] (3.4,0) -- (3.4,3);
\draw[black,dashed] (0+3.4,2) -- (2+3.4,0);
\draw[black,dashed] (3+3.4,1) -- (1+3.4,3);

\draw[step=0.2cm,gray,very thin] (6.8,0) grid (9.8,3);
\fill[black] (9.6,2.8) rectangle (9.8,3);
\fill[red] (0.4+6.8,0.8) rectangle (0.6+6.8,1);
\draw[gray,very thin] (6.8,0) -- (6.8,3);
\draw[black,dashed] (0+6.8,2) -- (2+6.8,0);
\draw[black,dashed] (3+6.8,1) -- (1+6.8,3);
\end{tikzpicture}
}
\caption{The moving process}
\end{figure}
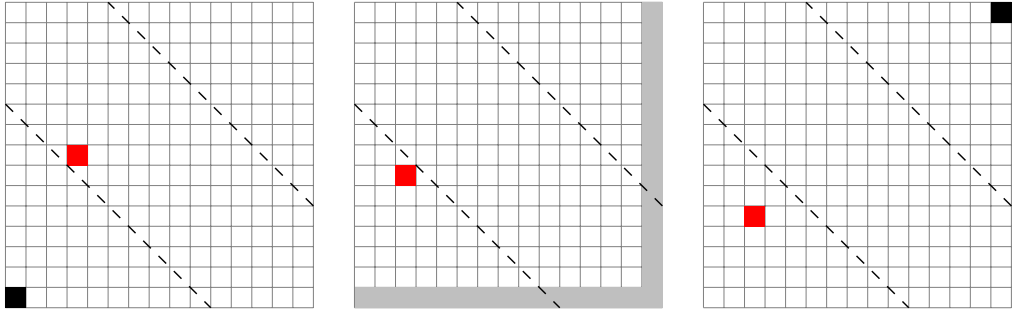

\end{proof}

This idea can be improved in many ways (e.g. one could simply take a $16$ by $16$ square with two holes). These all improve the upper bound very slightly. We would be surprised if there was a way to beat $n^{1.9}$ using a simple adaptation of this strategy and very surprised if there was a way to beat $n^{1.8}$.
\\
\\
We end this section with showing that Theorem 1.1 and Theorem 1.2 can be transplanted to Covering on the Torus. In fact, we show that Covering on the Torus is essentially the same as Covering on the Grid.

\begin{theorem}
Let $f(n)$ denote the number of cops needed to win Covering on the Grid, $[n]^2$. Let $g(n)$ denote the number of cops needed to win Covering on the Torus, $[n]^2$. Then $\frac{1}{4}f(n) \leq  g(n) \leq 4f(n)$. 
\end{theorem}

\begin{proof}
Suppose the $f(n)$ cops have a winning strategy on the grid. Then $4f(n)$ cops can win on the torus as follows: $4$ teams of $f(n)$ cops catch four robbers $R_{11}$, $R_{12}$, $R_{21}$ and $R_{22}$. These four phantom robbers are the original robber, the original robber reflected in line $x = (n+1)/2$, the original robber reflected in the line $y = (n+1)/2$ and the original robber reflected in both these lines. $R_{11}$ is chosen to be the phantom robber in the bottom left quadrant, $R_{12}$ is chosen to be the phantom robber in the bottom right quadrant and so on.
\\
\\
We see that these phantom robbers move as real robbers playing on the grid, so each team of cops can catch their phantom robber. We also see that at all times one of the phantom robbers is the real robber so the cops can win on the torus with $4f(n)$ cops. Similarly, if $g(n)$ cops can win on the torus, we consider the game played on the grid with $4g(n)$ cops, each cop being given $3$ additional cops, positions defined as with the phantom robbers in the paragraph above. This allows each cop to move as if they were on the torus, hence they can catch the robber - as the robber has strictly fewer move choices than on the torus.
\end{proof}

It is clear how this proof adapts to higher dimensions, with $4$ being replaced with $2^d$.

\section{Higher dimensions}
In this short section we first take a slight detour to a different question that ends up being very useful for both Covering in $[n]^d$ and Rugby. We consider the question of a robber in $\mathbb{Z}^d$ at the origin, and the cops choosing their starting position. The cops aim to catch the robber at time $T$. The robber has speed $2$.

\begin{theorem}
There exist constants $c_1$ and $c_2$, dependent on $d$, such that the number of cops needed to catch a robber at time $T$ in $d$ dimensions is between $c_1T^{d/2}$ and $c_2T^{d/2}$ for all $T$.
\end{theorem}

\begin{proof}
To prove the lower bound we will show by induction that for $T = 4^{n}-1$ the robber can beat $2^{d(n-1)}$ cops. The base case $n=1$ is simple. Suppose it is true for $n-1$. Consider the $2^d$ quadrants of $\mathbb{Z}^d$ and pick the one with the least cops in it. We then walk a distance of $\frac{3}{2}4^n$ into it along the diagonal, taking steps of length 2. We can then check that the only cops that can possibly catch us are the cops originally in the quadrant. We have $4^{n-1}-1$ time left and there are at most $2^{d(n-2)}$ cops that can catch us, so we can induct.
\\
\\
For the upper bound we'll again induct. We prove it in one dimension as it multiplies up to higher dimensions. The cops' strategy uses the following fact: a given cop starting at $c_0$, with the robber starting at $r_0$, can ensure that they end up at $\lceil(c_0+r_1)/2\rceil$ in time $T$ so long as $|\lceil(c_0+r_1)/2\rceil - c_0| \leq T$, where $r_1$ is the position of the cop at $T$.
\\
\\
The way the cop can achieve this is to always take one step toward $\lceil(c_0+r)/2\rceil$, where $r$ is the current position of the robber. Then when the cop reaches that point they just take the steps needed to stay on $\lceil(c_0+r)/2\rceil$. This is possible as $\lceil(c_0+r)/2\rceil$ changes by at most $1$ each robber move. If the cop is not able to ever reach $\lceil(c_0+r)/2\rceil$ then $\lceil(c_0+r)/2\rceil$ was always to the left or right of the cops' position so they took all the steps in the same direction, hence $|\lceil(c_0+r_1)/2\rceil - c_0| > T$
\\
\\
To prove the upper bound, we show that for $T = 4^n$, and for cops linearly spread out on $[-3T,3T]$ with common difference $2^{n-1}$, the cops can always catch the robber. The base case $n=1$ is immediate. Each cop follows the above strategy for $3T/4$, and we have that all the cops initially in $[r_1-3T/2, r_1+3T/2]$ have moved to $\lceil(c_0+r_1)/2\rceil$. As the initial position of each of the cops was linearly spread, so are these now with common difference $2^{n-2}$. They fully cover $[r_1-3T/4, r_1+3T/4]$, and we have $4^{n-1}$ time remaining. This completes the induction, using at most $6\times 2^{n+1} +1$ cops for $T = 4^n$.
\\
\\
In $d$ dimensions we take $(6\times 2^{n+1} +1)^d$ cops, the product of all the individual co-ordinate's strategies. This completes the proof.
\end{proof}

Our next corollary will show how this can be used to get bounds for Covering in $[n]^d$. The upper bound this gives is the best bound we have been able to find for dimensions above $2$, adapting Theorem 1.2 gives a worse bound. Theorem 1.1 can be simply translated into some higher dimensions to give a better lower bound than Corollary 2.2. Though as $d$ goes to infinity Corollary 2.2 gives a better lower bound.

\begin{corollary}
There exists $c_1$ and $c_2$, dependent on $d$, such that the number of cops needed to win Covering on $[n]^d$ is between $c_1n^{\frac{d}{2}}$ and $c_2 n^{\frac{d}{2}+1}$.
\end{corollary}

\begin{proof}
The lower bound is instant as the robber just ensures that at time $n/4$ no cops are on him by using the fixed time strategy.
\\
\\
The cops' strategy requires some more thought. We split the cops into $5n/4$ teams, each with enough cops to catch the robber in time $n/4$. Team $i$ is tasked with being on top of the robber at times congruent to $i$ modulo $5n/4$, without loss of generality $n$ a multiple of $4$. (At some points this strategy has cops outside the grid, but we can just consider the closest point on the grid to the cop and move the cop there).
\\
\\
Teams $i$ for $0 \leq i < n/4$, set up to catch the robber at time $i$ and catch the robber at that time. They then move to starting positions to catch a robber starting at the origin at time $n/4$, which takes them time at most $n$. Then at time $n + i$ they play the fixed time strategy to catch a phantom robber at time $5n/4 + i$. This phantom robber starts at the origin and chases the real robber, tracing the real robber's steps afterwards. When they catch this phantom robber, the phantom robber has caught the real robber so in fact they catch the real robber at time $5n/4 + i$. They then repeat to catch at $2(5n/4)+i,...$. The other teams simply set up for a robber starting at the origin and perform the same strategy. This ensures a cop is on the robber at all possible times and uses at most $c_2 n^{\frac{d}{2}+1}$ cops, for some $c_2$.
\end{proof}

This finishes the bounds we have been able to find for this question. In Section 4, we ask some natural follow on questions. The last thing to remark on is the unimportance of speed 2. All these proofs can easily be adapted to a quicker robber, giving different bounds. For high enough speeds Theorem 1.1 becomes worse than Corollary 2.2. But importantly, Theorem 1.2 will always show that the answer is $o(n^2)$. We'll see in the next section that the robber's speed becomes more important for Rugby.

\section{Rugby}
Intuitively, we may think that the best way to block the robber from winning Rugby is for all cops to remain with the same first coordinate and reduce it to Covering on the Grid, but we begin with another corollary of Theorem 2.1 which gives us a better strategy than this. In this section, we take $(0,0,0)$ as the centre of the `tunnel' $\mathbb{Z} \times [n]^d$ (without loss of generality $n$ is odd).

\begin{corollary}
There exists $c$ such that the number of cops needed to win Rugby in $d$ dimensions is at most $cn^{\frac{d+1}{2}}$.
\end{corollary}

\begin{proof}
We simply just use enough cops to catch the robber in time $\lfloor n/4\rfloor$ starting at $(0,0,0)$. If the robber's $x$ co-ordinate is ever 0 then we can employ the same strategy as in Corollary 2.2 to catch the robber. Therefore we prevent the robber from ever passing $x$ co-ordinate $\lfloor n/2 \rfloor$.
\end{proof}

This gives us a way of beating the naive strategy of $n^d$ cops occupying $\{0\}\times [n]^d$ in dimensions 2 and higher. For high enough speeds this uses closer and closer to $n^{d+1}$ cops; luckily we can revert to our strategies in Theorem 1.2 to beat this naive strategy. This raises the question of whether $o(n)$ cops can win $1$ dimensional Rugby. Our next theorem, proved using a similar method to Theorem 1.2,  shows that this is the case.

\begin{theorem}
We have that $O(n^{0.999})$ cops can win the game of Rugby in 1 dimension.
\end{theorem}

\begin{proof}
We prove that for a vertical side length of $10000^n$, we need only $2\times9900^n$ cops to catch the robber or prevent the robber from passing $x=0$. We allow the cops to go above and below the tunnel, as this adapts to a valid strategy as in Corollary 2.2. 
\\
\\
\begin{figure}[H]
\centering
\scalebox{1.5}{
\begin{tikzpicture}
\draw[step=0.2cm,gray,very thin] (0,0) grid (6,6);
\fill[red] (0.4,4.6) rectangle (0.6,4.8);
\fill[gray] (5.8,4.6) rectangle (6,5.4);
\fill[gray] (5.8,3.6) rectangle (6,4.4);
\fill[gray] (5.8,2.6) rectangle (6,3.4);
\fill[gray] (5.8-0.6,4.6) rectangle (6-0.6,5.4);
\fill[gray] (5.8-0.6,3.6) rectangle (6-0.6,4.4);
\fill[gray] (5.8-0.6,2.6) rectangle (6-0.6,3.4);
\fill[gray] (5.8-0.6,5.4) rectangle (6,5.6);
\fill[gray] (5.8-0.6,5.4-3) rectangle (6,5.6-3);
\end{tikzpicture}
}
\caption{Snapshot of stage 1 (not to scale)}
\end{figure}
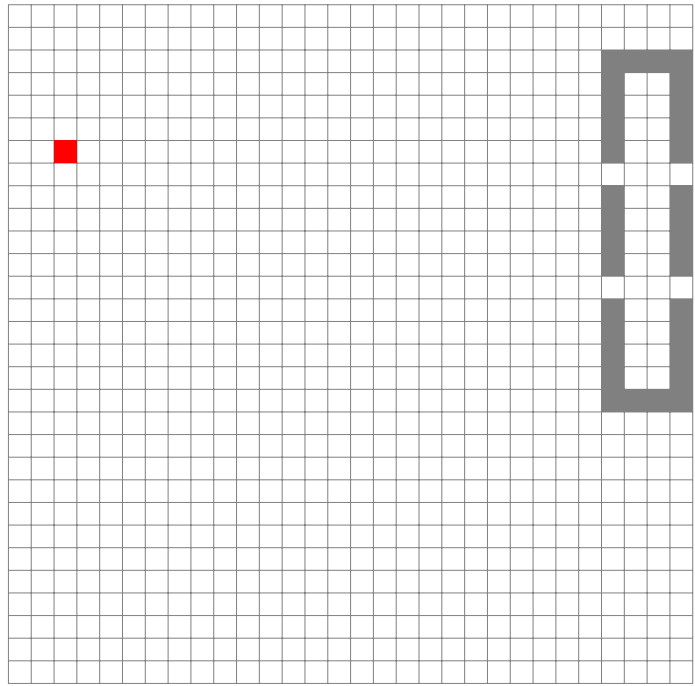
We'll walk through $n=1$ first. We create a net of cops, a rectangle of vertical side length $5050$ and horizontal side length 30. This initially has its right edge on $x = 0$ and has the centre of the net with $y$ co-ordinate $\lfloor \frac{n}{4} \rfloor$ (as the robber starts at $(-C,\lfloor\frac{n}{2}\rfloor)$). The cops move up and down to ensure the centre of this net has $y$ co-ordinate at $\lfloor r/2 \rfloor$ where $r$ is the $y$ co-ordinate of the robber. This ensures that the net always has one cop 25 above the robber and one cop 25 below the robber.
\\
\\
If the robber never reaches $x$ co-ordinate greater than $-6$ then we've won. If it ever does so (ignoring the fact it would have been caught by the left side of the net), it's trapped in our net, and we have stopped it from ever getting past $x=0$.
\\
\\
We now inductively claim that, for the game of Rugby with side length $10000^n$, with $2\times 9900^n$ cops, we can ensure that the robber never passes $x = -6\times 10000^{n-1}$, or the robber is trapped in a net of cops, i.e. the convex hull of a collection of cops. The base case is done.
\\
\\
The net is a rectangle of vertical side length $5050\times 10000^{n-1}$ and horizontal side length $30\times 10000^{n-1}$. Each sub-square has enough cops for the strategy with $n-1$. The Theorem 1.2 type idea is that once every 28 sub-squares on the vertical edges we have a gap where we have shadow cops. Therefore this uses less than $2\times9900^n$ cops. Each sub-square is playing the strategy `facing-in', trapping the robber inside the net.
\\
\\
In Figure 4 the red indicates the robber and the grey the cops. In stage 1, we use the exact same strategy as already described to ensure that if the robber ever crosses $x=-6\times 10000^{n-1}$ it is in the net and has $25$ sub-squares above and below it.

\begin{figure}[H]
\centering
\begin{tikzpicture}
\draw[step=0.2cm,gray,very thin] (0,0) grid (10,2);
\fill[red] (0,0.2) rectangle (0.2,0.4);
\fill[black] (0,0) rectangle (1.2,0.2);
\fill[black] (1.4,0) rectangle (6.8,0.2);
\fill[black] (7,0) rectangle (10,0.2);

\draw[step=0.2cm,gray,very thin] (0,0-3) grid (10,2-3);
\fill[red] (0.8,0.2-3) rectangle (1,0.4-3);
\fill[black] (0,0-3) rectangle (1.2,0.2-3);
\fill[gray] (1.2,0-3) rectangle (10,0.2-3);

\draw[step=0.2cm,gray,very thin] (0,0-6) grid (10,2-6);
\fill[red] (1.6,0.2-6) rectangle (1.8,0.4-6);
\fill[black] (0,0-6) rectangle (6.8,0.2-6);
\fill[black] (7,0-6) rectangle (10,0.2-6);
\end{tikzpicture}
\caption{Rotated snapshots of stage 2}
\end{figure}
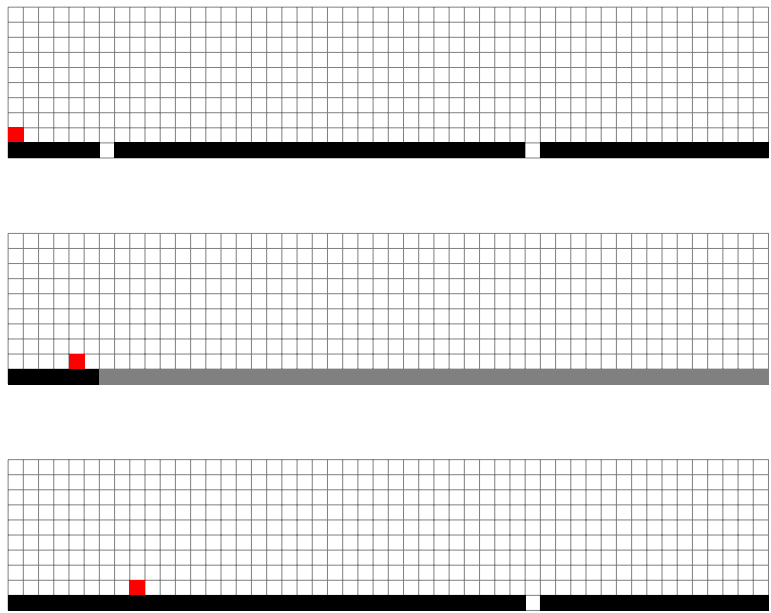

Stage 2 begins if the robber ever crosses $x=-6\times 10000^{n-1}$. Now, if the robber is ever 6 sub-squares away from an empty square, then we shift the empty sub-square to the top or the bottom as in Theorem 1.2, taking the route further away from the robber (this is the use of having a net; we are able to transport the cops around with the action taking place away from the robber). Then if the robber ever moves more than 12 away from the sub-square we undo this, again taking the route further away from the robber. Also, finally, note we will not have any empty squares on the top $25$ or bottom $25$ squares to ensure that no squares are empty when the robber reaches the top or bottom.
\\
\\
The gaps of 28 sub-squares ensures that we never have two strategies happening at once. We note that it takes potentially time $3\times 10000^{n-1}$ to fully complete the move as some cops will have to move across an empty sub-grid, so the robber can only move 6 sub-squares during a strategy playing out. We then have that if the robber were ever to escape the larger net, it would have to cross a side, therefore being $6 \times 10000^{n-2}$ away from the edge of the net and getting trapped in a smaller net. So the induction is complete.
\end{proof}
In comparison with Theorem 1.2, we need to be more careful with the final step. It could be the case that all sub-squares did their job blocking the shadow robber from getting past, but in fact each sub-square caught the shadow robber at a point where the robber was not there. Theorem 3.2 side-steps this issue as there is a global condition on the robber's co-ordinates when it gets trapped in the smaller net.
\\
\\
Like Theorem 1.2, this can also be improved slightly. Again we do not think we can take it that far. What is interesting is how dependent on speed 2 this is: to apply stage 2 we need to have some net that loops back on itself. To achieve this we needed stage 1 to use only $\sim n$ cops, and this breaks for higher speeds, needing $\sim 2n(s-1)/s$ cops. 
\\
\\
We think it will be hard to show that regardless of the robber's speed, $o(n)$ cops suffices to win Rugby, using a Theorem 1.2 type argument. We need to use the second dimension in order to be able loop the cops around, and this introduces redundancy so we naturally use more than $n$ cops and the fixes we have found for this all rely on the robber having some small speed.
\\
\\
Now we move on to a lower bound for 1 dimensional Rugby. We have to be much more careful than the other bounds in this paper as it's hard to get a handle on the cops' positions. We can easily see how a $\ln(n)$ type bound can work: the robber deals with the cops on much different time scales $k,k^2,k^3...$ (enumerating them $1,2,..,i$) running straight to the right and up/down away from the cop $i$ when it's $k^i$ away. The robber gives priority to the smaller timescales, and so long as $k$ is large enough the deviation possibly caused by $1,2,...,i-1$ cannot possibly effect $i$ too much. We use this idea in Theorem 3.3.
\\
\\
What we do is run straight to the right and have intervals $I_1,I_2,...$ at a distance $k,k^2,...$ away from the robber, each of width $2.2k^i$. We care only about the cops that could have a chance of catching the robber - so only count a cop as being in $I_i$ if there is a sequence of moves where he catches the robber. If a big chunk of cops are in $I_i$ together, say more than $t^i$ cops, then we can halve the number of these cops that reach $I_{i-1}$ by moving up or down. Using this we can say that if there were more than $t^{i+1}$ cops in $I_i$ then there were more than $t^{i+2}$ cops in $I_{i+1}$ etc.. If this holds then we prove a $\sim n^{(\ln(t)/\ln(k))}$ lower bound. For this we need the width of the $I_i$ to scale correctly so that cops together in $I_i$ had to be together in $I_{i+1}$.
\\
\\
We need some measure of how many requests we can get from $I_{j}$, $j<i$, when we are performing a request from $I_i$. We also need to ensure we cannot get another request from $I_i$ whilst we are already performing one. We do both of these by looking back and saying there would have been too many cops in some previous interval. One final thing we need to ensure is that we never go above or below the tunnel. We can do this by counting how many requests we can possibly have in total. In the theorem below, `we' will be in reference to the robber.
\begin{theorem}
In the 1-dimensional game of Rugby, we need $\Omega(n^{0.03})$ cops, for sufficiently large $n$.
\end{theorem}
\begin{proof}
We do this for general $k$, and fix it at the end. We have $t=\sqrt{3/2}$. Our robber starts far enough to the left such that all cops are at least $n$ to the right of the robber, and in the centre of the tunnel. At all times the robber moves to the right $2$ squares. We define intervals $I_1,I_2,...,I_m$ of cops as follows: $I_i$ is the collection of cops that have $x$ co-ordinate between $r_x + 0.9k^i$ and $r_x+ 3.1k^i$, where $r_x$ is the robber's current $x$ co-ordinate and with $y$ co-ordinate such that it is possible that the cop could reach the robber. We take $m = \lfloor\log_k{\frac{n}{2}}\rfloor$. Finally, note that the $x$ co-ordinate of a cop changes by between $-1$ and $-3$ relative to the robber after both the cops and the robber have moved.
\\
\\
We then get a `request' from $i$ whenever there are at least $\frac{1}{2}t^i$ cops in $I_i$ that are not currently `bad'. We now call these cops `bad' - the $1/2$ is there just to ensure that we never have cops in $I_1$. To complete the request from $i$ we need to move up or down so that half the bad cops have no chance to reach the robber regardless of the robber's and cops' moves. We also need to achieve this before the cop reaches $I_{i-1}$. Take a cop at $x$ co-ordinate $r_x+k^i$ and currently below the robber. If we ensure that by the time the cop's $x$ co-ordinate reaches $r_x+3.5k^{i-1}$ - say it took $t$ time - we have taken $3t/4+3k^{i-1}$ steps up(*), then we have lost that cop. 
\\
\\
So what we do is that if there are more bad cops beneath the robber, then for the next $3.1k^i$ time steps we move up unless we get a request from $j < i$. Similarly, if there are more bad cops above the robber we move down if possible. If we have no request currently then we remain in the same $y$ co-ordinate.
\\
\\
We say the system breaks if there are more than $\frac{1}{2}t^{i+1}$ cops in the interval $I_i$, enlarged to $[r_x + 0.5k^i,r_x+3.5sk^i]$. If the system never breaks then the robber wins, as the robber could never have been caught by a cop as no cops are ever in enlarged $I_1$. So we look at the first time the system breaks; say at $I_i$. What we'll show is that the system would have broken at $I_{i+1}$, so long as $i \neq m$. For $i = m$, we would have to have $\frac{1}{2}t^{m+1}$ cops. So proving this shows that the robber can escape less than $\frac{1}{2}t^{m+1}$ cops.
\\
\\
\begin{figure}[H]
\centering
\begin{tikzpicture}
\draw (0,-4) rectangle (8,4);
\fill[red] (0,-0.02) rectangle (0.04,0.02);
\draw (0.1,-0.05) -- (0.1,0.05) -- (0.12,0.06) -- (0.12,-0.06) -- cycle;
\draw (0.1*4,-0.05*4) -- (0.1*4,0.05*4) -- (0.12*4,0.06*4) -- (0.12*4,-0.06*4) -- cycle;
\draw (0.1*16,-0.05*16) -- (0.1*16,0.05*16) -- (0.12*16,0.06*16) -- (0.12*16,-0.06*16) -- cycle;
\draw (0.1*64,-0.05*64) -- (0.1*64,0.05*64) -- (0.12*64,0.06*64) -- (0.12*64,-0.06*64) -- cycle;
\end{tikzpicture}
\caption{Snapshot of four consecutive intervals, not to scale}
\end{figure}
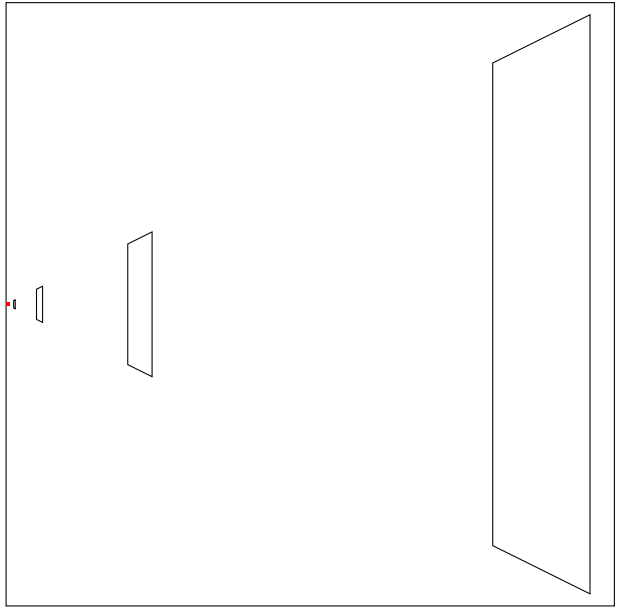
We first show that prior to this we never had two requests from a $j$ overlapping. Suppose that we did, then we would have had $\frac{1}{2}t^j$ cops in $[r_x+0.9k^j,r_x+3.1k^j]$, and we would have received another request in the next $3.1k^j$ time steps, and therefore there were $t^j$ cops in $[r_x+0.9k^j,r_x+12.4k^j]$. We look back time $0.5k^{j+1}-0.9k^{j}$. These cops were all in the interval $[r_x+0.5k^{j+1},r_x+1.5k^{j+1}+9.7k^{j}]$. Hence for $k$ sufficiently large, we have that they were all in the enlarged $I_{j+1}$ together, and therefore the system would have broken earlier. This gives a contradiction.
\\
\\
Next we show that before this all requests were completed. Take a request from $j$; we know that the smallest $t$ can be in (*) is $\frac{1}{3}(0.9k^j-3.5k^{j-1})$ and the largest is $3.1k^{j}$, so if we ensure that we move in the wrong direction at most $\frac{9}{120}k^j - 3k^{j-1}$ times due to requests from $h<j$ during the request from $j$, we have that we completed the request from $j$. If we have more than $13t^{j-h}$ requests from $h$ during a period of time $3.1k^j$, then we have that $2t^{j-h}$ happened in a period of time $0.5k^{j}$. Hence if we look back a time of $0.5k^{j}-0.5k^{h}$ from the first occurrence, we get that $2t^{j-h}\frac{1}{2}t^h$ cops were in the enlarged $I_j$; this gives a contradiction. So we have that the number of requests from $h$ is bounded by $13t^{j-h}$. Since this causes at most $3.1k^h$ moves in the wrong direction we need:
\begin{align*}
\sum_{h=1}^{j-1}40.3t^{j-h}k^h &< \frac{9}{120}k^j - 3k^{j-1}\\
40.3tk^{j-1} \sum_{n=0}^{\infty} (t/k)^{n}&< \frac{9}{120}k^j - 3k^{j-1}\\
40.3t &< (\frac{9k}{120}-3)(1-\frac{t}{k})
\end{align*}
\\
\\
So now we can show that the system had to break at $I_{i+1}$. We have $x$ cops in the enlarged $I_i$, $x>\frac{1}{2}t^{i+1}$. Looking $k^{i+1}-k^{i}$ earlier, they were all in $I_{i+1}$ earlier, so long as $k > 50$. We know that none of them were lost, but since $x>\frac{1}{2}t^{i+1}$, they would have triggered a request and been halved, so they must have had some that were bad already and had triggered a request with say $y$ other cops. We know $x+y < \frac{1}{2}t^{i+3}$ or else as they are in the enlarged $I_{i+2}$ together they would have broken the system; but as we halve the bad cops and we lost none of the $x$, all the lost ones must have been from the $y$, so $y>\frac{1}{4}t^{i+1}$. This gives a contradiction due to our choice of $t$.
\\
\\
We now choose $k$ appropriately. We choose $k = 700$, which gives us a lower bound of order $n^{\frac{\ln(3/2)}{2\ln(700)}}$. Note that $0.03$ is less than this exponent. To finish off the proof we need to show that we never moved up or down too much. We count how many requests we could have got from $i$. There were at most $n^{0.03}$ cops, and each time we get a request from $I_i$ it uses $\frac{1}{2}t^i$ cops, so we could have had at most $2n^{0.03}/t^i$ requests from $I_i$. Each request moves us $6k^i$ at most, therefore we moved at most $n^{0.03}\sum_{i=0}^{i=m} 12(k/t)^i$ steps up or down, which is $o(n)$. Hence for sufficiently large $n$, we need at least $n^{0.03}$ cops.
\end{proof}

As the current method stands, there is only small improvement to be had at higher speeds by modifying the constants used in the proof. In particular this does not give, as the speed approaches $\infty$, an exponent that approaches 1. Note that a speed 2 robber promising to sprint directly to the right can always be caught by $\sim n^{0.5}$ cops using the fixed time strategy. 
\\
\\
In the same vein, there is little to be gained by increasing the dimension. We can now cut down the bad cops by a factor of $2^d$, but we still have to bound the number of requests from earlier intervals and we have the issue that the bad cops we got rid off were all part of the $y$.

\section{Catching a Fast Robber}

In this short section, we show how these previous two theorems can be adapted to Catching a Fast Robber studied in \cite{BALISTER2017341} and \cite{ogfastrobber}. In \cite{BALISTER2017341}, the authors showed that for a robber with speed faster than $10^{25}$ the robber can win against $e^{c_s\log(n)/\log(\log(n)}$ for some constant $c_s$ depending on the speed s. They also showed that $\frac{n(2s-2)}{2s-1}+ O(1)$ cops can win.
\begin{corollary}
$O(n^{0.999})$ cops can catch a robber with speed 2 on the grid $[n]^2$, for $n$ sufficiently large.
\end{corollary}
\begin{proof}
We use the same number of cops, i.e $2\times 9100^n$ for a $10000^n$ size grid. We modify stages 1 and 2. In stage 1, the net starts at the centre of the grid with longer side along the $y$ axis. We then walk up/down so that the $y$ co-ordinate is on $\lfloor r_y/2\rfloor$. We then see if the robber is in the right or left hand side of the grid - without loss of generality on the left. The strategy is now identical to before but the net also moves one to the left each time step, forcing the robber to eventually start stage 2.
\\
\\
In stage 2 we now move the bottom or top of the net up or down to shrink the net when that edge is not involved in a move. When we have shrunk the net enough we can get rid of all the empty squares to the outside. So now we just continually shrink the net, uniformly along each side, forcing the robber to eventually get caught in one of the smaller nets and we can induct.
\end{proof}
Now we move on to the lower bound. We show that a robber of speed 2 can continually escape $n^{0.03}$ cops. Our strategy is to run around the grid playing as if it were the tunnel. We have to be careful in checking that we can choose a good place to start and that we can change direction.
\begin{corollary}
A robber with speed $2$ can evade capture from $\Omega(n^{0.03})$ cops on the grid $[n]^2$, for $n$ sufficiently large.
\end{corollary}
\begin{proof}
$k$, $t$ and $m$ are as before. We've already shown that we move up or down $o(n)$, therefore we can assume that we do not move up or down more than $n/6$ when moving a distance of $n$ through the tunnel. Say the bottom left corner has co-ordinate $(0,0)$, then we aim to start somewhere in the $n/6$ by $n/6$ square that has bottom left co-ordinate $(0,2n/3)$. We then sprint through the tunnel to the right, until we get to $x$ co-ordinate $2n/3$ and then we aim to turn down and repeat.
\\
\\
What suffices to initialise our strategy is that at most $\frac{1}{4}t^i$ cops are in the $2k^i$ by $2k^i$ square centred on the robber. This is because, if this does hold, imagine only the cops that can possibly catch the robber when he starts sprinting to the right. These cops could have stayed in their positions whilst the robber sprinted from a distance $n$ away towards them, playing the strategy described in Theorem 3.3. As at most $\frac{1}{4}t^i$ end up being a distance $2k^i$ away from the robber for all $i$, no request was given.
\\
\\
Let $C_i$ denote the number of squares that are banned due to there being more than $\frac{1}{4}t^i$ cops on the $2k^i$ by $2k^i$ square around it. We consider the $16k^{2i}$ sub-lattices with common difference $4k^i$. By the pigeonhole principle one of these has more than $C_i/16k^{2i}$ banned squares in it. Then we must have at least $C_it^{i}/64k^{2i}$ cops, therefore $C_i < 64n^{0.03}k^{2i}/t^{i}$. Therefore, we have that $\sum_{i=1}^{i=m}C_i < 100n^{0.03}k^{2m}/t^{m} < n^2/36$ for sufficiently large $n$. 
\\
\\
We can then sprint to the right using the strategy given in Theorem 2.3. When we get to $x$ co-ordinate $2n/3$ we aim to turn down. At each step, if we can we start sprinting down, we do. We consider $D_i$, now the number of $x$ co-ordinates banned due to there being more than $\frac{1}{4}t^i$ cops in the $2k^i$ by $2k^i$ square centred on it at the time the robber was there. We consider the $8k^i$ sub-lattices, now 1 dimensional, and we have that $D_it^i/8k^i < n^{0.03}$. So summing we get that $\sum_{i=1}^{i=m}D_i < 10n^{0.03}k^{m}/t^{m}$. Hence, for sufficiently large $n$, less than $n/6$ $x$ co-ordinates are banned. So we can start sprinting down the right hand side of the grid.
\\
\\
We repeat this strategy forever, turning clockwise around the square. Note that we never leave the square as we turn between $2n/3$ and $5n/6$ of the way to the other side and we do not move more than $n/6$ up or down as we go through the tunnel. Hence, for sufficiently large $n$, the robber with speed $2$ can evade capture on the $n$ by $n$ grid.
\end{proof}

\section{Open Questions}
We do not believe that either of our bounds are optimal for Covering on $[n]^2$, we also do not have an exact guess where the truth lies. We do think that it is likely closer to $n^2$ than $n$. As said before, we can certainly improve Theorem 1.2 with an adaptation of the ideas. Potentially, getting into the weeds of the cops moving to their new shadow cops' position could be fruitful. But, being very optimistic, even if it were possible to do the argument in a 3 by 3, this would only improve the upper bound to  $n^{1.89}$. So we do not see getting the bounds much closer by adapting Theorem 1.2.
\\
\\
We think to drastically improve the lower bound we would need a new idea, but this we are less confident of. In fact our original proof gave us $n^{1.05}$, so we have been able to nudge it up quite far. As commented earlier, we would need to modify the path or the sub-shapes to improve it further.
\\
\\
Perhaps the most natural question about Covering in the $d$ dimensional grid is whether the behaviour as $d$ goes to $\infty$ is like the lower or upper bound, or neither.

\begin{question}
Is there a constant $c_1 > 0$ such that for all dimensions, for sufficiently large $n$, we need more than $n^{\frac{d}{2}+c_1}$ cops to win Covering on $[n]^d$? Is there a constant $c_2 > 0$ such that for all dimensions, for sufficiently large $n$, we need at most  $n^{\frac{d}{2}+1-c_2}$ cops to win?
\end{question}

For Rugby, an `effective' lower bound really should have the property that for fixed $d$ and speed tending to infinity, the exponent tends to $d$. This seems quite challenging with the freedom the cops have. We also think it would be very interesting to find an upper bound that did not have an argument dependent on speed, especially if it used a different idea to Theorem 1.2 and Theorem 3.2. We strongly suspect the two conjectures below to be true. If the first were false, that would imply a `critical' speed where the structure of the game changes completely, which seems unlikely.

\begin{conjecture}
In the 1 dimensional game of Rugby, for all speeds of the robber, $o(n)$ cops suffices to win.
\end{conjecture}

\begin{conjecture}
In the $d$ dimensional game of Rugby, for all $\epsilon > 0$, we have that there exists a speed $s$ such that a robber of speed $s$ can beat $n^{d -\epsilon}$ cops for sufficiently large $n$
\end{conjecture}

One last interesting question is whether Rugby is `equivalent' to catching a fast robber. It seems that any sensible strategy for one should definitely be adaptable to the other, but we cannot see any systematic way of achieving this. 
\\
\\
However, if one has a cop strategy for catching a fast robber that promises to have caught the robber in $O(n)$ time, then we can see how we could adapt it to a strategy for Rugby. Conversely, suppose that the cops can win Rugby with a strategy that promises to prevent the robber passing $x=0$, then with a bit more thought we could adapt this into a strategy that catches the robber in the $n$ by $n$ grid.
\begin{question}
Do there exists constants $c_1$ and $c_2$, such that $c_1f(n) \leq g(n) \leq c_2f(n)$; where $f(n)$ is the number of cops needed to win 1 dimensional Rugby and $g(n)$ is the number of cops to catch a fast robber in the grid (where both games are played with a speed 2 robber)?
\end{question}

\section*{Acknowledgements}
The author would like to thank Trinity College Cambridge for funding the research.
\bibliographystyle{amsplain}
\bibliography{ACoveringPursuitGame}
\end{document}